\documentclass{amsart}
\usepackage{hyperref}
\usepackage[utf8]{inputenc}
\usepackage{listings}
\usepackage[bottom=1in, top=1in, left=1.25in, right=1.25in]{geometry}
\usepackage{todonotes}
\usepackage{xcolor}
\usepackage{comment}
\usepackage{graphicx}
\usepackage{url}

\newtheorem{thm}{Theorem}[section]
\newtheorem{lemma}[thm]{Lemma}

\newtheorem{cor}[thm]{Corollary}

\newtheorem{conjecture}[thm]{Conjecture}
\newtheorem{prop}[thm]{Proposition}

\newtheorem{obs}[thm]{Observation}

\newtheorem{definition}[thm]{Definition}

\theoremstyle{remark}
\newtheorem{remark}{Remark}

\usepackage{listings}
\usepackage{color}

\definecolor{dkgreen}{rgb}{0,0.6,0}
\definecolor{gray}{rgb}{0.5,0.5,0.5}
\definecolor{mauve}{rgb}{0.58,0,0.82}

\title{Products of reflections in smooth Bruhat intervals}
\date{\today}

\author{Christian Gaetz}
\thanks{C.G. was partially supported by an NSF Mathematical Sciences Postdoctoral Research Fellowship under grant no. DMS-2103121}
\address{Department of Mathematics, Harvard University, Cambridge, MA 02138.}
\email{\href{mailto:gaetz@math.harvard.edu}{gaetz@math.harvard.edu}.}

\author{Ram K. Goel}
\address{Portland, OR.}
\email{\href{mailto:ram.krishna.goel@gmail.com}{ram.krishna.goel@gmail.com}.}

\begin{document}
\maketitle

\begin{abstract}
A permutation is called smooth if the corresponding Schubert variety is smooth. Gilboa and Lapid prove that in the symmetric group, multiplying the reflections below a smooth element $w$ in Bruhat order in a \emph{compatible order} yields back the element $w$. We strengthen this result by showing that such a product in fact determines a saturated chain $e \to w$ in Bruhat order, and that this property characterizes smooth elements. 
\end{abstract}

\section{Introduction}

\label{Intro Section}
A permutation $w$ in the symmetric group $S_n$ is called \emph{smooth} if the corresponding Schubert variety $X_w$ is smooth. This class of permutations is very well-studied, and there are several famous criteria for smoothness in terms of Kazhdan--Lusztig polynomials, Poincar\'{e} polynomials of Schubert varieties, and Bruhat graphs \cite{Carrell} and permutation pattern avoidance \cite{pattern-avoidance}. Recently, there has been renewed interest \cite{self-dual, Gilboa_2021, Lapid-tightness} in the structure of smooth permutations and their relation to Bruhat order. We continue this study by refining a result from \cite{Gilboa_2021} and providing a new criterion for smoothness.

For an element $w \in S_n$, let $\mathbf{C}_{\mathcal{T}}(w)$ be the set of reflections in $S_n$ which lie in the Bruhat interval $[e,w]$. When $w$ is smooth, Gilboa and Lapid \cite{Gilboa_2021} defined a set of total orders on $\mathbf{C}_{\mathcal{T}}(w)$ called \emph{compatible orders} (see Definition~\ref{def:compatible-order}); they proved:

\begin{thm}[Theorem 1.2 of \cite{Gilboa_2021}]
\label{Gilboa_Main_Thm}
Let $w \in S_n$ be smooth, then a compatible order on $\mathbf{C}_{\mathcal{T}}(w)$ exists, and for all compatible orders $t_1\prec t_2\prec \cdots \prec t_k$ we have $t_1t_2\cdots t_k = w$. 
\end{thm}

Our main theorem refines this result by showing that these factorizations determine saturated chains in the Bruhat order and that smooth permutations are characterized by this stronger property.

\begin{thm}
\label{thm:intro-main}
A permutation $w \in S_n$ is smooth if and only if there exists an ordering $t_1\prec t_2\prec \cdots \prec t_k$ of $\mathbf{C}_{\mathcal{T}}(w)$ such that 
\[
e\to t_1 \to t_1t_2 \to \cdots \to t_1\cdots t_k = w
\]
is a saturated chain in Bruhat order. Furthermore, if $w$ is smooth, any compatible order satisfies this property.
\end{thm}

Theorem~\ref{thm:intro-main} is proven in Section~\ref{Main_Theorem_Section}.

\begin{remark}
Just the existence of an ordering $t_1\prec t_2\prec \cdots \prec t_k$ of $\mathbf{C}_{\mathcal{T}}(w)$ such that $t_1\cdots t_k=w$, without the saturated Bruhat chain condition from Theorem~\ref{thm:intro-main}, does not guarantee smoothness. For example, the permutation $w=35142$ is not smooth, but 
\[
\mathbf{C}_{\mathcal{T}}(w)=\{T_{1,2},T_{1,3},T_{2,3},T_{2,4},T_{3,4},T_{2,5},T_{4,5},T_{3,5}\},
\]
and 
\[
w=T_{1,2}T_{1,3}T_{2,3}T_{2,4}T_{3,4}T_{2,5}T_{4,5}T_{3,5},
\]
where $T_{i,j}$ denotes the reflection swapping $i,j$.
\end{remark}

\section*{Acknowledgements}

We wish to thank Pavel Etingof, Slava Gerovitch, and Tanya Khovanova, the organizers of the MIT PRIMES-USA program, during which this research was conducted. We are also grateful to Erez Lapid for helpful correspondence and to the anonymous referee for their careful corrections to an earlier version of this article. 

\section{Background}
\label{Background Section}

Much of this section reviews results and definitions from the work \cite{Gilboa_2021} of Gilboa and Lapid. We assume the reader is familiar with the (strong) Bruhat order $(S_n, \leq)$ on the symmetric group (see, for example \cite{Bjorner} for definitions and basic facts).

\subsection{Admissible sets and compatible orders}

Fix a positive integer $n\ge 1$. Let $\mathcal{T}=\{T_{i,j}:1\le i<j\le n\}$ be the set of reflections in $S_n$, and define $\mathcal{C}^{2,3} = \mathcal{T} \cup \{R_{i,j,k}, L_{i,j,k}:i<j<k\}$, where $R_{i,j,k}=T_{i,j}T_{j,k}$ and $L_{i,j,k}=T_{j,k}T_{i,j}$ are 3-cycles. Define $\mathbf{C}(w) = \{ \tau \in \mathcal{C}^{2,3}:\tau \le w\}$ and $\mathbf{C}_{\mathcal{T}}(w)=\{\tau\in \mathcal{T}:\tau\le w\}$. 

Define $\mu_w:[n]\to [n]$ for $w\in S_n$ as 
\[ \mu_w(i) = \max\{w(1),\ldots,w(i)\}.\]
It is useful to note the following identity:
\begin{equation} \label{eq:reflection-less-bruhat}
T_{i,j} \le w \iff \mu_w(i)\ge j \text{ and } \mu_{w^{-1}}(i) \ge j.
\end{equation}  

By \cite{Carrell}, a permutation $w\in S_n$ is smooth if and only if
\begin{equation}
\label{eq:smooth-in-terms-of-length}
  \ell(w)=|\mathbf{C}_{\mathcal{T}}(w)|,
\end{equation}
where $\ell(w)$ denotes the Coxeter length of $w$. A famous alternative characterization was given by Lakshmibai--Sandhya \cite{pattern-avoidance}: an element of $S_n$ is smooth if and only if it avoids the patterns $3412$ and $4231$.

\begin{definition}[Admissible sets, Definition 2.1 of \cite{Gilboa_2021}]
A subset $A\subseteq \mathcal{C}^{2,3}$ is called \emph{admissible} if:
\begin{itemize}
    \item $A$ is downward closed with respect to Bruhat order, 
    \item if $R_{i,j,l},L_{i,k,l}\in A$ with $i<j,k<l$, then $T_{i,l}\in A$, and
    \item whenever $T_{i,j},T_{j,k}\in A$ with $i<j<k$, at least one of $R_{i,j,k}$ and $L_{i,j,k}$ is in $A$. 
\end{itemize}
\end{definition}

By Lemma 2.2 of \cite{Gilboa_2021}, $\mathbf{C}(w)$ is admissible for smooth $w$.

\begin{definition}[Compatible orders, Section 1.2 of \cite{Gilboa_2021}]
\label{def:compatible-order}
A linear order $\prec$ on $A_{\mathcal{T}}=A\cap \mathcal{T}$ for an admissible set $A$ is called \emph{compatible} (with $A$) if the following condition holds: if $T_{i,j},T_{j,k}\in A$ for some $i<j<k$, then
\begin{itemize}
    \item if $T_{i,k}\in A$, then $T_{i,j}\prec T_{i,k}\prec T_{j,k}$ or $T_{j,k}\prec T_{i,k}\prec T_{i,j}$, and 
    \item if $T_{i,k}\not \in A$, then $R_{i,j,k}\in A$ if and only if $T_{i,j}\prec T_{j,k}$. 
\end{itemize}
\end{definition}

\subsection{Existence of compatible orders}

In this subsection, we outline Gilboa and Lapid's \cite{Gilboa_2021} proof that for any admissible subset $A\subseteq \mathcal{C}^{2,3}$, there exists a compatible order.

\begin{definition}[Definition 3.2 of \cite{Gilboa_2021}] 
\label{Wedge_definition}
Suppose that $A\subseteq \mathcal{C}^{2,3}$ is admissible. We say that $T_{i,j}$ is a \emph{wedge} for $A$ if: 
\begin{itemize}
    \item $T_{i,j}\in A$,
    \item $T_{i-1,i}\not \in A$, and
    \item $R_{i,j,j+1}\not \in A$. 
\end{itemize}
\end{definition}

Let $w\in S_n$ be smooth. Assume a wedge $T_{i,j}$ exists for $A=\mathbf{C}(w)$; in particular, this means that $T_{i,j}\in A$, so $T_{i,j} \le w$. The criterion (\ref{eq:reflection-less-bruhat}) implies $T_{i,r} \in A$ for all $i<r\le j$. Letting $A_{\mathcal{T}}=\mathbf{C}_{\mathcal{T}}(w)$, define 
\[ A^\circ = A \setminus (\{ T_{i,r}:r>i\} \cup \{ L_{i,r,l} ,R_{i,r,l}: l>r>i\}),\]
and let $A_{\mathcal{T}}^\circ = (A^\circ)\cap \mathcal{T}$, so that
\[ A_{\mathcal{T}}^\circ = A_{\mathcal{T}} \setminus \{T_{i,r} : i<r\le j\}.\]
By Lemma 3.10a of \cite{Gilboa_2021}, $A^\circ$ is admissible, and hence by induction a compatible order $\prec^\circ$ exists on $A_{\mathcal{T}}^\circ$. Then the compatible order for $A_{\mathcal{T}}$ is constructed by adding on $T_{i,j},T_{i,j-1},\ldots,T_{i,i+1}$ in that order to the end of the compatible order $\prec^\circ$ on $A_{\mathcal{T}}^\circ$. Lemma \ref{Compatible_order_construction} implies that the resulting order is also compatible:

\begin{lemma}[Lemma 4.5 of \cite{Gilboa_2021}]
\label{Compatible_order_construction}
Suppose that $\emptyset \not = A \subseteq \mathcal{C}^{2,3}$ is admissible and $T_{i,j}$ is a wedge for $A$. Then, any compatible order $\prec^\circ$ for $A^\circ$ may be extended to a compatible order $\prec$ on $A$ by requiring that
\[ T_{k,l} \prec T_{i,j}\prec T_{i,j-1}\prec \cdots \prec T_{i,i+1}\]
for each $T_{k,l}\in A_{\mathcal{T}}^\circ$. 
\end{lemma}

Inductively applying Lemma \ref{Compatible_order_construction} to any admissible set $A$ provides a construction for a compatible order for $A_{\mathcal{T}}$. 

\begin{cor}[\cite{Gilboa_2021}]
\label{Exists_compatible_order_for_C_w}
For all smooth $w\in S_n$,  $\mathbf{C}_{\mathcal{T}}(w)$ has a compatible order.
\end{cor}

Together with Theorem~\ref{Gilboa_Main_Thm}, this implies that for any smooth $w \in S_n$, there is an ordering $t_1 \prec \cdots \prec t_k$ of $\mathbf{C}_{\mathcal{T}}(w)$ such that $t_1 \cdots t_k=w$.

\subsection{More background facts} 

\begin{lemma}[Lemma 2.1.4 of \cite{Bjorner}]
\label{Bjorner_bruhat_covering}
Let $x,y\in S_n$. Then $x$ is covered by $y$ in Bruhat order if and only if $y=xT_{a,b}$ for some $a<b$ such that $x(a)<x(b)$ and there does not exist any $c$ for which $a<c<b$ and $x(a)<x(c)<x(b)$. 
\end{lemma}

\begin{obs}[Observation 3.3 of \cite{Gilboa_2021}]
\label{Observation_swapping_i}
If $T_{i,j}$ is a wedge for the admissible set $A\subseteq \mathcal{C}^{2,3}$, then $\{T\in A_{\mathcal{T}} : T(i)\not = i\} = \{T_{i,r}\}_{r=i+1}^j$. 
\end{obs}

\begin{lemma}[Lemma 3.8a of \cite{Gilboa_2021}]
\label{Lemma_3.8_Lapid}
Let $w\in S_n$ be such that $\mathbf{C}(w)$ is admissible, and let $1\le i<j\le n$. Then $T_{i,j}$ is a wedge for $\mathbf{C}(w)$ if and only if $w([i-1])=[i-1]$ and $w(i)\ge j = w^{-1}(i)$. 
\end{lemma}

\section{Preliminary Lemmas} 
\label{Lemmas_Section} 

This section proves some preliminary results needed for the proof of Theorem~\ref{thm:intro-main}. Throughout this section let $w\in S_n$ be smooth, and suppose there exists a wedge $T_{i,j}$ for the admissible set $\mathbf{C}(w)$. 

\begin{lemma}
\label{first_right_side_lemma}
If $w(i)>j$, then there exists an index $j'$ with $j+1\le j' \le n$ such that $w(j) < w(j') < w(i)$. 
\end{lemma}
\begin{proof}

For the sake of contradiction, assume no such $j'$ exists. Then $w(j') > w(i)$ or $w(j') < w(j)$ for all $j+1\le j' \le n$. Suppose the latter occurred for some $j'$. Then $w(j') < w(j)=i$. But we know $\{w(1),\ldots,w(i-1)\} = \{1,\ldots,i-1\}$ by Lemma \ref{Lemma_3.8_Lapid}, so $j' < i$, contradiction. Hence for all $j' \in [j+1,n]$, we have $w(j') > w(i) > j$, which means $w(j') \in \{j+2,\ldots,n\}$, contradicting the injectivity of $w$.
\end{proof}

\begin{lemma}
\label{lem:wedge-implies-decreasing}
We have $w(i)>w(i+1)>\cdots>w(j)$.
\end{lemma}
\begin{proof}
By Lemma \ref{Lemma_3.8_Lapid}, we know $\{w(1),\ldots,w(i-1)\}=\{1,\ldots,i-1\}$ and $w(i)\ge j$ and $w^{-1}(i)=j$, so $w(j)=i$. 

We claim that $w(i+1),\ldots,w(j-1) \le w(i)$. Suppose $w(i')>w(i)$ for some $i'\in (i,j)$. Note $i=w(j)\le j\le w(i)$. Suppose $w(i)>j$, then by Lemma \ref{first_right_side_lemma}, there exists some $j'\in [j+1,n]$ for which $w(j) < w(j') < w(i)$. Now, $i<i'<j<j'$, and 
\[ w(j)<w(j')<w(i)<w(i'). \]
Therefore, $w$ contains a $3412$ pattern, contradicting smoothness. Otherwise, we have $w(i)=j$; the assumption $w(i')>w(i)$ implies that some $k \in \{i+1,\ldots,j-1\}$ is not among $\{w(j-1),\ldots,w(i+1)\}$ and since $w([i-1])=[i-1]$, taking $j'=w^{-1}(k)$ again yields an occurrence of $3412$ in positions $i<i'<j<j'$. Thus in either case $w(i') \le w(i)$ for all $i'\in (i,j)$. 

We claim that $w(i+1),\ldots,w(j) \ge i$. This is because $\{w(1),\ldots,w(i-1)\} = \{1,\ldots,i-1\}$ by Lemma \ref{Lemma_3.8_Lapid}, so we must have $w(i')\ge i$ for all $i' \ge i$. So in fact $w(i') > i$ for all $i' \in (i,j)$ since $w(j)=i$. 

Now, if we ever have $w(i') < w(j')$ for some $i<i'<j'<j$, then 
\[ w(i) > w(j') > w(i') > w(j)=i,\]
contradicting $w$ avoiding the 4231 pattern. Therefore, $w(i) > w(i+1)>\cdots>w(j)$. 
\end{proof}

\begin{lemma}
\label{w'_ordering}
Consider a smooth $w$. If $T_{i,j}$ is a wedge for the admissible set $\mathbf{C}(w)$, and 
\[ w'=wT_{i,i+1}T_{i,i+2}\cdots T_{i,j},\]
then 
\[ w'(i+1)>w'(i+2)>\cdots>w'(j)>w'(i).\]
\end{lemma}
\begin{proof}
For each $1\le d \le j-i$, let
\[ w_d=wT_{i,i+1}T_{i,i+2}\cdots T_{i,i+d}. \]
We claim
\[
w_d(i+1)>w_d(i+2)>\cdots>w_d(i+d) > w_d(i)  >w_d(i+d+1)>w_d(i+d+2)>\cdots>w_d(j).     
\]
Induct on $d$. For the base case of $d=1$, since we know $w(i)>w(i+1)>\cdots>w(j)$ by Lemma~\ref{lem:wedge-implies-decreasing}, $w_1=wT_{i,i+1}$ satisfies
\[ w_1(i+1)>w_1(i)>w_1(i+2)>\cdots>w_1(j),\] 
as needed. Now notice that $w_{d+1}=w_dT_{i,i+d+1}$. Swapping $w_d(i)$ and $w_d(i+d+1)$ in the inequality chain for $d$ gives
\[
w_d(i+1)>w_d(i+2)>\cdots>w_d(i+d) > w_d(i+d+1)>w_d(i)>w_d(i+d+2)>\cdots>w_d(j).
\]
This completes the induction. Finally, we plug in $d=j-i$: notice $w'=w_{j-i}$, so
\[ w'(i+1)>w'(i+2)>\cdots>w'(j)>w'(i).\]
\end{proof}

For the remainder of this section let $t_1\prec \cdots \prec t_k$ be a compatible order on $\mathbf{C}_{\mathcal{T}}(w)$; such an order exists by Corollary \ref{Exists_compatible_order_for_C_w}. Define: 
\begin{itemize}
    \item $w'=T_{i,j}T_{i,j-1}\cdots T_{i,i+1}$,
    \item $q_d=t_d t_{d+1}\cdots t_k w'$, for each $1\le d\le k$, and
    \item $w_d'=T_{i,i+d}T_{i,i+d-1}\cdots T_{i,i+1}$, for each $1\le d \le j-i$. 
\end{itemize}

We omit the elementary proof of the following observation:

\begin{obs}
\label{Tail_properties} \text{}
\begin{enumerate}
    \item[(a)] For all $1\le d \le j-i$, $\ell(w'_d)=d$. 
    \item[(b)] $(w'_{j-i}(i),w'_{j-i}(i+1),\ldots,w'_{j-i}(j))=(i+1,i+2,\ldots,j,i)$. 
\end{enumerate}
\end{obs}

\begin{obs}
\label{No_swap_less_i_more_i}
None of $t_1,\ldots,t_k$ are of the form $T_{x,y}$, where $x<i\le y$. 
\end{obs}
\begin{proof}
Suppose $t_\ell = T_{x,y}$ for some $\ell \in [1,k]$ and for some $x<i \le y$. By identity (1.7a) of \cite{Gilboa_2021} we have $T_{i-1,i} \le T_{x,y}$ and so, since $\mathbf{C}_{\mathcal{T}}(w)$ is downward closed, we have $T_{i-1,i}\in \mathbf{C}_{\mathcal{T}}(w)$, contradicting the fact that $T_{i,j}$ is a wedge. 
\end{proof}

\begin{lemma}
For all $1\le d\le k$, we have
\begin{align*}
    \{q_d(1),\ldots,q_d(i-1)\} &= \{1,\ldots,i-1\}, \\
    \{ q_d(i),\ldots,q_d(n)\} &= \{i,\ldots,n\}. 
\end{align*}
\end{lemma}

\begin{proof}
By Observation \ref{Tail_properties}, notice
\[ (w'(1),\ldots,w'(n))=(1,\ldots,i-1,i+1,\ldots,j,i,j+1,\ldots,n),\]
and also note by Observation \ref{No_swap_less_i_more_i} that none of $t_1,\ldots,t_k$ swap a pair of numbers one of which is less than $i$ and the other greater than or equal to $i$. Hence the numbers $1,\ldots,i-1$ will be in the first $i-1$ positions for all $q_d$, since this is the case for $w'$, and successively applying $t_k$, then $t_{k-1}$, and so on will not change this.
\end{proof}

\section{Proof of Theorem~\ref{thm:intro-main}}
\label{Main_Theorem_Section}

The following result gives part of Theorem~\ref{thm:intro-main}.

\begin{thm}
\label{Our_Main_Thm}
For any smooth $w\in S_n$, there exists a compatible order $\prec$ on $\mathbf{C}_{\mathcal{T}}(w)$ with $t_1\prec t_2\prec \cdots \prec t_k$ satisfying the following conditions: 
\begin{itemize}
    \item $t_1t_2\cdots t_k = w$. 
    \item $e\to t_1 \to t_1t_2 \to \cdots \to t_1\cdots t_k=w$ is a saturated chain in Bruhat order. 
    \item $e\to t_k \to t_kt_{k-1} \to \cdots \to t_k\cdots t_1=w^{-1}$ is a saturated chain in Bruhat order.  
\end{itemize}
\end{thm}

Let $w\in S_n$ be smooth; we saw in Section \ref{Background Section} that $A=\mathbf{C}(w)$ is admissible. As stated in Remark 3.4 of \cite{Gilboa_2021}, at least one of $A$ or $A^{-1}$ has a wedge. Observation 4.3 of \cite{Gilboa_2021} implies that the reverse of a compatible order on $A$ is a compatible order on $A^{-1}=\mathbf{C}_{\mathcal{T}}(w^{-1})$, the admissible set corresponding to the smooth permutation $w^{-1}$. Note that the second saturated chain property in Theorem \ref{Our_Main_Thm} is the same as the first saturated chain property if we reverse the order of $t_1,\ldots,t_k$. Therefore, we can without loss of generality assume that $A$ has a wedge. 

\begin{prop}
\label{Saturated_chain_prefix}
Let $w\in S_n$ be smooth, and let $t_1\prec \cdots \prec t_k$ be the compatible order of $\mathbf{C}_{\mathcal{T}}(w)$ described prior to Corollary \ref{Exists_compatible_order_for_C_w}. Then
\[ e\to t_1\to t_1t_2\to \cdots \to t_1\cdots t_k\]
is a saturated chain in Bruhat order. 
\end{prop}

\begin{proof}
We will prove that for all $0\le a \le k$, 
\[ e\to t_1\to t_1t_2\to \cdots \to t_1\cdots t_a\] 
is a saturated chain in Bruhat order by induction on $a$. The base case of $a=0$ is clear since the chain consists only of $e$. 

Let $A=\mathbf{C}_{\mathcal{T}}(w)$, and let $T_{i,j}$ be a wedge for $A$. Let $w'$ be the product of the reflections in $A_{\mathcal{T}}^\circ$ in the order $\prec^\circ$. By the inductive hypothesis, the prefix products of reflections in the order $\prec^\circ$ satisfy the saturated chain condition in Bruhat order. It hence suffices to show (due to the construction of the compatible order provided in Lemma \ref{Compatible_order_construction}) that
\[ w'\to w'T_{i,j} \to w'T_{i,j}T_{i,j-1}\to \cdots \to w'T_{i,j}T_{i,j-1}\cdots T_{i,i+1}\]
is a saturated chain in the Bruhat order. 

By Lemma \ref{w'_ordering}, 
\[ w'(i+1)>w'(i+2)> \cdots >w'(j) > w'(i). \]
For each $i+1\le d \le j$, define
\[ w_d = w'T_{i,j}T_{i,j-1}\cdots T_{i,d}.\]
We claim
\[
w_d(i+1)>w_d(i+2)>\cdots>w_d(d-1) >w_d(i)>w_d(d)>w_d(d+1)>\cdots>w_d(j). 
\]
Proceed by induction on $d$, the base case of $d=j$ being true since $w_0=w'T_{i,j}$ satisfies 
\[ w_0(i+1)>w_0(i+2)>\cdots>w_0(j-1)>w_0(i)>w_0(j) \]
simply by swapping $w'(i)$ and $w'(j)$ in the inequality chain from Lemma \ref{w'_ordering}. For the inductive step, since $w_{d+1}=w_dT_{i,d-1}$, swapping $w_d(i)$ and $w_d(d-1)$ gives
\[
w_d(i+1)>w_d(i+2)>\cdots>w_d(d-2) > w_d(i)>w_d(d-1)>w_d(d)>w_d(d+1)>\cdots>w_d(j),
\]
as desired. 

Finally, we want to show for each $d$ that $w_d \to w_dT_{i,d-1}$ is a Bruhat covering relation. This is equivalent to
\[  w_d(i)<w_d(d-1) \text{ and }w_d(i+1),\ldots,w_d(d-2) \not \in (w_d(i), w_d(d-1)),\]
by Lemma \ref{Bjorner_bruhat_covering}. 
Both of these statements hold since we proved
\[ w_d(i)>w_d(i+2)>\cdots>w_d(d-2)>w_d(d-1)>w_d(i).\]
\end{proof}

\begin{prop}
\label{Saturated_chain_suffix}
Let $w\in S_n$ be smooth, and let $t_1\prec \cdots \prec t_k$ be the compatible order of $\mathbf{C}_{\mathcal{T}}(w)$ of Corollary \ref{Exists_compatible_order_for_C_w}. Then
\[ e\to t_k\to t_kt_{k-1}\to \cdots \to t_k\cdots t_1\]
is a saturated chain in Bruhat order. 
\end{prop}
\begin{proof}
It suffices to prove that the length increases by exactly one at each step. Any element of $S_n$ and its inverse have equal lengths, so it is equivalent to show that
\[ e\to t_k \to t_{k-1}t_k \to \cdots \to t_1\cdots t_k\]
is a saturated Bruhat chain. We will show 
\[ e \to t_k \to t_{k-1}t_k\to \cdots \to t_{k-a+1}\cdots t_k\]
is a saturated Bruhat chain, for each $0\le a \le k$. In other words, $\ell(t_{k-a+1}\cdots t_k)=a$, for each $0\le a \le k$. 

Let $A=\mathbf{C}_{\mathcal{T}}(w)$. As per the construction in Lemma \ref{Compatible_order_construction}, if we let $t_1 \prec^\circ \cdots \prec^\circ t_k$ be the compatible order for $A^\circ$, then $\prec$ satisfies
\[ t_1\prec \cdots \prec t_k\prec T_{i,j}\prec T_{i,j-1}\prec \cdots \prec T_{i,i+1}. \]
Let $w'=T_{i,j}T_{i,j-1}\cdots T_{i,i+1}$, and let $w=t_1\cdots t_kT_{i,j}T_{i,j-1}\cdots T_{i,i+1}$. By induction, $e \to t_k \to t_{k-1}t_k\to \cdots \to t_1\cdots t_k$
is a saturated Bruhat chain. 

Firstly, from Observation \ref{Tail_properties}, we have
\begin{align*}
    \ell(T_{i,i+1})&=1, \\
    \ell(T_{i,i+2}T_{i,i+1})&=2, \\
    &\vdots \\
    \ell(T_{i,j}T_{i,j-1}\cdots T_{i,i+1})&=j-i, 
\end{align*}
so
\[ T_{i,i+1} \to T_{i,i+2}T_{i,i+1}\to \cdots \to T_{i,j}T_{i,j-1}\cdots T_{i,i+1}\]
is a saturated chain in Bruhat order. 

Finally, we prove that
\[ w' \to t_kw' \to t_{k-1}t_kw'\to \cdots \to t_1\cdots t_kw'\]
is a saturated chain in Bruhat order. By Observation \ref{Tail_properties}, 
\[ (w'(1),\ldots,w'(n))=(1,\ldots,i-1,i+1,\ldots,j,i,j+1,\ldots,n). \]
The key observation is that in the above, removing the value $i$ makes the remaining sequence an increasing sequence. Combining Observation \ref{No_swap_less_i_more_i} with the fact that all of $\{1,\ldots,i-1\}$ appear before $i$ for all $q_d$ and with the fact that $t_k\to t_{k-1}t_k\to \cdots t_1\cdots t_k$ is a Bruhat chain implies that the number of inversions increases by one in each step of $t_kw'\to t_{k-1}t_kw'\to \cdots \to t_1\cdots t_kw'$.
\end{proof}

Combining Propositions \ref{Saturated_chain_prefix} and \ref{Saturated_chain_suffix} along with Theorem \ref{Gilboa_Main_Thm} proves Theorem \ref{Our_Main_Thm}. In particular, note that the compatible order for which we proved the properties of Theorem \ref{Our_Main_Thm} is the compatible order constructed from Corollary \ref{Exists_compatible_order_for_C_w}. 

We proved above that the properties of Theorem \ref{Our_Main_Thm} hold for one particular compatible order on $\mathbf{C}(w)$, we now extend this statement to all compatible orders for $\mathbf{C}(w)$.

\begin{prop}
\label{theorem_all_compatible_orders}
The conditions from Theorem~\ref{Our_Main_Thm} hold for any compatible order $\prec$ on $\mathbf{C}_{\mathcal{T}}(w)$.
\end{prop}

For an admissible set $A\subseteq \mathcal{C}^{2,3}$, define the graph $\mathcal{G}_A$ as follows. The vertices are the the compatible orders on $A_{\mathcal{T}}$ and there is an edge between two compatible orders $\prec_1$ and $\prec_2$ if and only if $\prec_2$ can be obtained from $\prec_1$ by one of the following \emph{elementary operations}:
\begin{enumerate}
    \item Swapping two commuting transpositions that are adjacent in $\prec_1$, and
    \item Reversing the order of consecutive $T_{i,j}, T_{i,k}, T_{j,k}$ (with $i < j < k$) in $\prec_1$ to $T_{j,k}, T_{i,k}, T_{i,j}$, or vice versa. 
\end{enumerate}
Note that the product of the reflections in the order $\prec_2$ is the same as the product in the order $\prec_1$.

\begin{lemma}[Lemma 4.9 of \cite{Gilboa_2021}]
\label{graph_connected}
For admissible $A\subseteq \mathcal{C}^{2,3}$, the graph $\mathcal{G}_A$ is connected. 
\end{lemma}

\begin{prop}
\label{operation_1_preserves_chain}
Let $w'\in S_n$, and suppose $T_{i,j}$ and $T_{k,l}$ commute. If 
\[ w' \to w'T_{i,j} \to w'T_{i,j} T_{k,l} \]
is a saturated chain in Bruhat order, then so is
\[ w'\to w'T_{k,l} \to w'T_{k,l}T_{i,j}.\]
\end{prop}
\begin{proof}
Let $x=w'T_{i,j} T_{k,l}=w'T_{k,l} T_{i,j}$. By assumption, $w' < x$ and $\ell(x)-\ell(w')=2$; since Bruhat intervals of rank two are diamonds \cite{Bjorner}, there is a unique element $y \in [w',x]$ not lying on the saturated chain $w' \to w'T_{i,j} \to w'T_{i,j} T_{k,l}$. By results of Dyer \cite{Dyer}, $(w')^{-1}y$ is a reflection in the reflection subgroup of $S_n$ generated by $T_{i,j}$ and $T_{k,l}$. These two elements are the only reflections in the subgroup they generate, so $y=w'T_{k,l}$. 
\end{proof}

\begin{prop}
\label{operation_2_preserves_chain}
Let $w'\in S_n$ and $1\le i<j<k\le n$. Then if
\[ w'\to w'T_{i,j} \to w'T_{i,j}T_{i,k} \to w'T_{i,j}T_{i,k}T_{j,k}\]
is a saturated chain in Bruhat order, then so is
\[ w'\to w'T_{j,k} \to w'T_{j,k}T_{i,k} \to w'T_{j,k}T_{i,k}T_{i,j}.\]
\end{prop}
\begin{proof}
Let $x=w'T_{i,j}T_{i,k}T_{j,k}=w'T_{j,k}T_{i,k}T_{i,j}$. By assumption, $w' < x$ and $\ell(x)-\ell(w')=3$. The reflections $T_{i,j}, T_{j,k},$ and $T_{i,k}$ generate a reflection subgroup of $S_n$ isomorphic to $S_3$, and so by \cite{Dyer}, the interval $[w',x]$ is isomorphic to an interval in $S_3$. Since this interval has rank three, it must be isomorphic to the whole Bruhat order on $S_3$, with isomorphism mapping $w' \cdot u \mapsto \hat{u}$ where for a permutation $u$ of $\{i,j,k\}$, we write $\hat{u}$ for the corresponding permutation of $\{1,2,3\}$. Then, since the proposition clearly holds for $n=3$, we are done.
\end{proof}

By Propositions \ref{operation_1_preserves_chain} and \ref{operation_2_preserves_chain}, we conclude that the elementary operations preserve the saturated chain conditions from Theorem~\ref{Our_Main_Thm}. Combining this with Lemma \ref{graph_connected} proves Proposition \ref{theorem_all_compatible_orders}.

\begin{proof}[Proof of Theorem~\ref{thm:intro-main}]
Proposition~\ref{theorem_all_compatible_orders} gives the ``if" direction of Theorem~\ref{thm:intro-main}. For the ``only if" direction, notice that if $t_1 \prec \cdots \prec t_k$ is an ordering of $\mathbf{C}_\mathcal{T}(w)$ such that 
\[
e \to t_1 \to t_1t_2 \to \cdots \to t_1\cdots t_k=w
\]
is a saturated chain in Bruhat order, then we have
\[
|\mathbf{C}_{\mathcal{T}}(w)|=k=\ell(w),
\]
so $w$ is smooth by (\ref{eq:smooth-in-terms-of-length}).
\end{proof}

\section{Conjectures for Generalizations to Other Weyl Groups}
\label{Conjectures_D_n_Section}

We refer the reader to \cite{Humphreys} for background on root systems and Weyl groups. We use the convention that the Weyl group $W$ of type $D_n$ has roots, simple roots, and positive roots given respectively by:
\begin{align*}
    R &= \{e_j+e_i : j>i\} \cup \{e_j-e_i:i\not = j\} \cup \{-e_j-e_i:j>i\}, \\
    \Pi &= \{e_{i+1}-e_{i}:1\le i\le n-1\} \cup \{e_2+e_1\}, \\
    R^+ &= \{e_j-e_i:j>i\} \cup \{e_j+e_i:j>i\}. 
\end{align*}

\begin{definition}
Given a root system $R$ with positive roots $R^+$, the \emph{root poset} $(R^+, \leq)$ is defined as follows: for $\alpha, \beta \in R^+$ we have $\alpha$ covered by $\beta$ if and only if $\beta - \alpha \in \Pi$.
\end{definition}

\subsection{Conjectured Type $D$ analog of admissible sets}

The first object we attempt to generalize the notion of an admissible set. Define $\mathcal{C}^{2,3} = \{t_\alpha:\alpha \in R^+\}\cup \{t_\alpha t_\beta:\alpha,\beta,\alpha+\beta \in R^+\}$. 

\begin{definition}
Define the function $f:R^+\to \Pi$ by
\begin{align*}
    f(e_j-e_i) &= e_j-e_{j-1} \\
    f(e_j+e_i) &= \begin{cases} e_2+e_1 &\text{ if } (j,i)=(2,1), \\
    e_j-e_{j-1} &\text{ otherwise}. \end{cases}
\end{align*}
\end{definition}
 
\begin{lemma}
If $\alpha,\beta\in R^+$ and $\alpha+\beta \in R^+$, then $f(\alpha)\not = f(\beta)$. 
\end{lemma}
\begin{proof}
Suppose $f(\alpha)=f(\beta)$. Then we must have $\alpha = e_j \pm e_i$ and $\beta = e_j \pm e_{i'}$ for some $j>i,i'$. Then $\alpha+\beta$ has a $2e_j$ as a term in it, so it cannot be a root.
\end{proof}

\begin{definition}[Conjectured generalization of admissible sets]
\label{generalization_admissible_sets}
A subset $A\subseteq \mathcal{C}^{2,3}$ is called \emph{admissible} if the following conditions hold:
\begin{itemize}
    \item $A$ is downward closed with respect to the Bruhat order. 
    
    \item Suppose $\alpha,\beta,\alpha',\beta'\in R^+$ satisfy
    \begin{itemize}
        \item $\alpha+\beta=\alpha'+\beta'=\gamma$ for some some $\gamma\in R^+$; and
        \item $f(\beta) \prec f(\alpha)$ and $f(\beta')\prec f(\alpha')$; and 
        \item $t_\alpha t_\beta \in A$ and $t_{\beta'}t_{\alpha'}\in A$. 
    \end{itemize}
    Then $t_{\alpha+\beta}\in A$. 

    \item Suppose $\alpha,\beta\in R^+$ satisfy $t_\alpha \in A$ and $t_\beta \in A$ and $\alpha+\beta \in R^+$. Then at least one of $t_\alpha t_\beta$ and $t_\beta t_\alpha$ is in $A$. 
\end{itemize}

\end{definition}

\begin{definition}[Conjectured generalization of compatible orders]
An ordering $\prec$ on the reflections of an admissible (as per Definition \ref{generalization_admissible_sets}) set $A\subseteq \mathcal{C}^{2,3}$ is called \emph{compatible} if the following condition holds. Whenever $\alpha,\beta\in R^+$ and $t_\alpha,t_\beta \in A$ with $\alpha+\beta\in R^+$, then:
\begin{itemize}
    \item If $t_{\alpha+\beta} \in A$, then either $t_\alpha \prec t_{\alpha+\beta} \prec t_\beta$ or $t_\beta  \prec t_{\alpha+\beta} \prec t_\alpha$. 
    \item If $t_{\alpha+\beta} \not \in A$, then $t_\alpha t_\beta \in A \iff t_\alpha \prec t_\beta$. 
\end{itemize}
\end{definition}

As for the symmetric group, we say an element $w \in W$ is \emph{smooth} if the Schubert variety $X_w$ is a smooth variety. These elements are again characterized, for example, by a notion of pattern avoidance, as elaborated in \cite{smooth_in_D_n}. Equipped with these definitions of admissible sets and compatible orders, we now state conjectures about generalizations of the results for smooth elements of $W$. 

\begin{conjecture}
\label{conj:type-D}
Let $w\in W$ be smooth and let $A=\mathbf{C}(w)$. Then:
\begin{enumerate}
    \item $A$ is admissible,
    \item a compatible order on $A_{\mathcal{T}}$ exists, and
    \item if we multiply the elements of $A_{\mathcal{T}}$ in any compatible order, we get $w$. 
\end{enumerate}
\end{conjecture}

Conjecture \ref{conj:type-D} has been verified for $W$ of type $D_4$.

\begin{remark}
Given a suitable definition of the function $f$, Conjecture~\ref{conj:type-D} could equivalently be made uniformly for all simply-laced finite types. In non-simply-laced type, complications arise because smoothness and \emph{rational smoothness} are no longer equivalent \cite{Carrell}.
\end{remark}

\bibliographystyle{plain}
\bibliography{revision}

\end{document}